\documentclass{llncs}

\usepackage{amssymb,stmaryrd,mathrsfs,dsfont,mathtools}

\newtheorem{notation}[theorem]{Notation}
\usepackage{enumitem}
\usepackage[ruled,linesnumbered]{algorithm2e}
\usepackage{algpseudocode}
\usepackage{tikz}
\tikzstyle{vertex}=[circle, draw, inner sep=0pt, minimum size=2.5pt, fill=black]
\newcommand{\vertex}{\node[vertex]}
 
\usetikzlibrary{decorations.markings}
\usetikzlibrary{decorations.pathmorphing}

\begin{document}

\frontmatter         

\pagestyle{headings}  

\mainmatter             

\title{Line Graphs of Non-Word-Representable Graphs are Not Always Non-Word-Representable}

\titlerunning{Line Graphs of Non-Word-Representable Graphs are Not Always Non-Word-Representable}

\author{Khyodeno Mozhui \and Tithi Dwary \and K. V. Krishna}

\authorrunning{Khyodeno Mozhui \and K. V. Krishna}

\institute{Department of Mathematics\\ Indian Institute of Technology Guwahati, India\\
	\email{k.mozhui@iitg.ac.in};
	\email{tithi.dwary@iitg.ac.in};
	\email{kvk@iitg.ac.in}}

\maketitle         

\begin{abstract}
	A graph is said to be word-representable if there exists a word over its vertex set such that any two vertices are adjacent if and only if they alternate in the word. If no such word exists, the graph is non-word-representable. In the literature, there are examples of non-word-representable graphs whose line graphs are non-word-representable. However, it is an open problem to determine whether the line graph of a non-word-representable graph is always non-word-representable or not? In this work, we address the open problem by considering a class of non-word-representable graphs, viz., Mycielski graphs of odd cycles of length at least five, and show that their line graphs are word-representable.	
\end{abstract} 


\section{Introduction}

A word $w$ over a finite set $X$ is a finite sequence of elements (letters) from the set $X$. A word is called $k$-uniform if each letter appears exactly $k$ times in the word.  A subword $u$ of $w$, denoted by $u \ll w$, is a subsequence of $w$. For $Y \subseteq X$, the subword of $w$ restricted to the elements of $Y$ is denoted by $w|_{Y}$. For instance, if $u = abdcbadac$, then $u|_{\{a, b, c\}} = abcbaac$. We say two letters $a$ and $b$ alternate in a word $w$ if $w|_{{\{a, b\}}}$ is of the form $abab\cdots$ or $baba\cdots$, which is of either even or odd length.

In this paper, we only consider the simple graphs, the graphs without loops or parallel edges. To distinguish between two-letter words and edges of graphs, we write $\overline{ab}$ to denote an undirected edge between the vertices $a$ and $b$. A directed edge from $a$ to $b$ is denoted by $\overrightarrow{ab}$.   An orientation of a graph is an assignment of direction to each edge so that the resulting graph is a directed graph. A directed graph on the vertex set $\{a_1,a_2,\ldots, a_n\}$, where $n \ge 4$, is called a shortcut if it contains a directed path, say $\langle a_1,a_2,\ldots, a_n\rangle$, and the edge $\overrightarrow{a_1a_n}$, called shortcutting edge, such that there is no edge $\overrightarrow{a_ia_j}$ for some $i, j$ with $ 1\le i<j \le n$. An orientation of a graph is said to be semi-transitive if the corresponding directed graph has no directed cycles and no shortcuts. 

A graph $G = (V, E)$ is called a word-representable graph if there exists a word $w$ over its vertex set $V$ such that, for all $a, b \in V$, $\overline{ab} \in E$ if and only if $a$ and $b$ alternate in $w$. It is evident that if $w$ represents a graph $G$, then its reversal also represents $G$. In \cite{kitaev08}, it was shown that every word-representable graph is represented by a  $k$-uniform word, for some $k$. Further, it was proved that if a $k$-uniform word $w$ represents $G$, then its cyclic shift\footnote{A cyclic shift of a word $w$ of the form $uv$ is $vu$ for some words $u$ and $v$.} also represents $G$. Since their introduction in \cite{kitaev08order}, word-representable graphs have been extensively studied. In fact, this notion generalizes several important graph classes including circle graphs, 3-colorable graphs, cover graphs, and comparability graphs.

From a computational perspective, word-representable graphs are interesting. For example, the maximum clique problem, which is NP-hard in general, can be solved in polynomial time for word-representable graphs. Also, determining whether a given graph is word-representable is NP-complete. Nonetheless, in \cite{halldorsson16},  word-representable graphs were characterized as the graphs which admit semi-transitive orientations. There is a long line of research in the literature where semi-transitive orientation has been used as a powerful tool for studying the word-representability of graphs, such as polyomino triangulations \cite{Akrobotu_2015}, (extended) Mycielski graphs \cite{hameed_2024,kitaev_mycielski}, split graphs \cite{split_graphs_2024}, triangle-free graphs \cite{triangle-free_2023} and  Kneser graphs \cite{kneser_graphs_2020}.

The line graph of a graph $G = (V,E)$, denoted by $L(G)$,  is a graph with $E$ as the vertex set and two vertices in $L(G)$ are adjacent if and only if their corresponding edges share a vertex in $G$. For $k \ge 1$, we write $L^{k+1}(G) = L(L^{k}(G))$, where $L^1(G) = L(G)$. In \cite{kitaev_linegraphs,kitaev_line2}, various problems were posed on the word-representability of line graphs. In this work, we address one of the long standing open problems and establish that line graphs of non-word-representable graphs are not always non-word-representable.

\section{Remarks on Open Problems}  

It was established in \cite{kitaev_linegraphs,kitaev_line2} that the line graph of the complete graph $K_4$ is word-representable, but for $n \ge 5$, the line graph of a complete graph $K_n$ is non-word-representable. 
 
\begin{figure}
	\centering
	\begin{minipage}[b]{.5\textwidth}
		\centering
		\begin{tikzpicture}[scale=0.8]
			\vertex (1) at (0,0) [label=below:$4$] {};  
			\vertex (2) at (-1.3,-1) [label=left:$2$] {}; 
			\vertex (3) at (1.3,-1) [label=right:$3$] {}; 
			\vertex (4) at (0,1.2) [label=above:$1$] {}; 		
			\vertex (5) at (1,1.5) [label=above:$y$] {}; 		
			
			\path
			(1) edge (2)
			(1) edge (3)
			(1) edge (4)
			(2) edge (3)
			(2) edge (4)
			(3) edge (4)
			(5) edge (4);
		\end{tikzpicture}	
		\caption{$K_4'$}
		\label{k4'}	
	\end{minipage}
	\begin{minipage}[b]{.4\textwidth}
		\centering
		\begin{tikzpicture}[scale=0.8]
			\vertex (1) at (0,1.2) [label=above:$a_{23}$] {};  
			\vertex (2) at (-0.6 ,0.1) [label=left:$a_{12}$] {}; 
			\vertex (3) at (0.6,0.1) [label=right:$a_{13}$] {}; 
			\vertex (4) at (0,-0.8) [label=below:$a_{14}$] {}; 	
			\vertex (5) at (-1.8,-1) [label=below:$a_{24}$] {};  
			\vertex (6) at (1.8,-1) [label=right:$a_{34}$] {}; 
			\vertex (7) at (0,-0.4) [label=above:$a_{1y}$] {}; 
			
			\path
			(1) edge (2)
			(1) edge (3)
			(2) edge (3)
			(2) edge (5)
			(2) edge (4)
			(3) edge (6)
			(3) edge (4)
			(5) edge (4)
			(4) edge (6)
			(2) edge (7)
			(3) edge (7)
			(4) edge (7)
			(1) edge  [bend right] (5)
			(5) edge  [bend right] (6)
			(1) edge  [bend left] (6) ;
		\end{tikzpicture}
		\caption{$L(K_4')$}
		\label{Lk4'}
	\end{minipage}
\end{figure}

Consider the graph $K_4'$ given in Fig. \ref{k4'}. Note that the line graph $L(K_4')$ is as per Fig. \ref{Lk4'}, in which the vertices $a_{ij}$ are corresponding to the edges $\overline{ij}$ in $K_4'$. It can be observed that $L(K_4')$ is isomorphic the graph co-$(T_2)$, which is non-word-representable by \cite[Theorem 2]{halldorsson10}. Hence, we have the following lemma.  

\begin{lemma}\label{K_4'_wr}
	The graph $L(K_4')$ is non-word-representable.	
\end{lemma}

Note that if $H$ is a subgraph of $G$, then $L(H)$ is an induced subgraph of $L(G)$. By hereditary property of word-representable graphs, all induced subgraphs of a word-representable graph are word-representable \cite{kitaev15mono}. From Lemma \ref{K_4'_wr}, we deduce that if $G$ contains $K_4'$ as a subgraph, then its line graph $L(G)$ is non-word-representable. In particular, since a complete graph $K_n$, for $n \ge 5$, contains $K_4'$ as a subgraph, the line graph $L(K_n)$ is non-word-representable. Thus, we have a shorter and alternative proof for the word-representability of line graphs of complete graphs given in \cite{kitaev_line2}.

\begin{remark}\label{con_1}
If $G$ is a connected graph on at least five vertices and contains $K_4$ as an induced subgraph then $L(G)$ is non-word-representable.
\end{remark}

Furthermore, it was shown in  \cite{kitaev_linegraphs,kitaev_line2} that if a graph $G$ is neither a cycle, a path nor the complete bipartite graph $K_{1,3}$, then the line graph $L^k(G)$ is non-word-representable, for $n \ge 4$. Let $G$ be a connected graph on six vertices contains vertex of degree at least four. Note that the complete bipartite graph $K_{1,4}$ is a proper subgraph of $G$ and hence $K_4$ is a proper subgraph of $L(G)$. Hence, by Lemma \ref{K_4'_wr}, we conclude  that $L^k(G)$ is non-word-representable, for $k\ge 2$.

In this paper, we address the following open problem posed in \cite{kitaev_linegraphs,kitaev_line2}.

\vskip 0.2cm

\noindent \textbf{Problem 1.} \textit{Is the line graph of a non-word-representable graph always non-word-representable?}
\vskip 0.2cm

Problem 1 was highlighted at multiple contexts (e.g., see \cite{kitaev_dlt_2017,kitaev_survey}) and there have been attempts to address the problem. For instance, in \cite{akbar_line,Akrobotu_thesis}, a negative answer to Problem 1 was attempted by giving the graph $A$ depicted in Fig. \ref{A} as counterexample. However, we show that the line graph of $A$ is non-word-representable. As a result, the problem remains open. Nevertheless, we settle the problem by showing that the line graph of the Mycielski graph of odd cycle $\mu(C_{2n+1})$, for $n \ge 2$, is  word-representable.

In \cite{akbar_line}, the non-word-representable graph $A$ given in Fig. \ref{A} (cf. \cite[Figure 3.9]{kitaev15mono}) was considered and shown that $L(A)$ is word-representable. Note that the graph $A$ contains $W_5'$ given in Fig. \ref{Wn'} as a subgraph, and thus $L(A)$ contains $L(W_5')$ as an induced subgraph. As word-representable graphs are hereditary, $L(W_5')$ must be word-representable if $L(A)$ is word-representable. However, we observe that $L(W_5')$ is non-word-representable showing that the graph in Fig. \ref{A} is not a counterexample to Problem 1.

\begin{figure}
	\centering
	\begin{minipage}[b]{.3\textwidth}
		\centering
		\begin{tikzpicture}[scale=1]
			\vertex (n+1) at (0,0) [label=below:$ $] {};  
			\vertex (3) at (-0.8,-1) [label=left:$ $] {}; 
			\vertex (4) at (0.8,-1) [label=right:$ $] {}; 
			\vertex (1) at (0,1.2) [label=above:$ $] {}; 	
			\vertex (2) at (-1.3,0.3) [label=left:$ $] {}; 
			\vertex (n) at (1.3, 0.3) [label=right:$ $] {}; 
			\vertex (0) at (0, -0.6) [label=right:$ $] {}; 
			\path
			(1) edge (2)
			(2) edge (3)
			(3) edge (4)
			(1) edge (n)
			(1) edge (n+1)
			
			(3) edge (n+1)
			(4) edge (n+1)
			
			(4) edge (n)
			(2) edge (0)
			(3) edge (0)
			(4) edge (0)
			(n) edge (0)
			;
		\end{tikzpicture}	
		\caption{$A$}
		\label{A}	
	\end{minipage}
	\begin{minipage}[b]{.3\textwidth}
		\centering
		\begin{tikzpicture}[scale=1]
			\vertex (n+1) at (-5,-0.5) [label=right:$  $] {};  
			\vertex (3) at (-5.8,-1.5) [label=left:$  $] {}; 
			\vertex (4) at (-4.3,-1.5) [label=right:$  $] {}; 
			\vertex (1) at (-5,0.7) [label=above:$  $] {}; 	
			\vertex (2) at (-6.3,-0.2) [label=left:$  $] {}; 
			\vertex (n) at (-3.7,-0.2) [label=right:$  $] {}; 
			\path
			(1) edge node[above] {$e_2$}  (2)
			(2) edge node[left] {$e_3$} (3)
			(3) edge  node[below] {$e_4$} (4)
			(1) edge node[above] {$e_1$}(n)
			(2) edge node[below] {$b_2$} (n+1)
			(3) edge node[right] {$b_3$}(n+1)
			(4) edge node[right] {$b_4$}(n+1)
			(1) edge node[left] {$b_1$} (n+1)
			(4) edge node[right] {$e_5$} (n)
			;
		\end{tikzpicture}	
		\caption{$W_5'$}
		\label{Wn'}	
	\end{minipage}
	\begin{minipage}[b]{.3\textwidth}
		\centering
		\begin{tikzpicture}[scale=0.9]
			\vertex (27) at (0,0) [label=above:$b_1$] {};  
			\vertex (37) at (1,0) [label=above:$b_2 $] {}; 
			\vertex (47) at (0,-1) [label=below :$b_4$] {}; 
			\vertex (57) at (1,-1) [label=below :$b_3$] {}; 
			
			\vertex (12) at (-1.1,0) [label=above left:$e_1$] {}; 	
			\vertex (23) at (0.5,1.1) [label=above:$e_2$] {}; 
			\vertex (34) at (2.1,-0.5) 
			[label=right:$e_3 $] {}; 
			\vertex (45) at (0.5, -2.1) 
			[label=below:$e_4 $] {}; 
			\vertex (15) at (-1.1, -1) [label=left:$e_5 $] {}; 
			
			\path
			(27) edge (37)
			(27) edge (37)
			(27) edge (47)
			(27) edge (57)
			(37) edge (47)
			(37) edge (57)	 
			(47) edge (57);
			
			\path 
			(27) edge (12)
			(27) edge (23)
			(37) edge (23)
			(37) edge (34)
			(57) edge (34)
			(47) edge (45)
			(57) edge (45)
			(15) edge (47);
			
			\path 
			(12) edge (23)
			(23) edge (34)
			(34) edge (45)
			(45) edge (15)
			(12) edge (15);
		\end{tikzpicture}	
		\caption{$L(W_5')$}
		\label{L(Wn')}	
	\end{minipage}			
\end{figure}

\begin{notation}
	In this, we assume a word $w$ is cyclic, i.e., $w$ is written on a circle.
	\begin{enumerate}
		\item 	The statement \index{$\exists (abca)$} $\exists (abca)$ is true for a word $w$ denotes that there is at least two consecutive occurrence of $a$ in the word $w$ such that the vertices $b$ and $c$ occur  between them and $b$ appears on the left side of $c$. 
		\item The statement \index{$\forall (abca)$} $\forall (abca)$ is true for a word $w$ denotes that between every two consecutive occurrences of the vertex $a$ in $w$, the vertices $b$ and $c$ occur between them and $b$ always appears on the left side of $c$. 
	\end{enumerate}
\end{notation}

The following two results are needed to prove that the graph $L(W_5')$ is non-word-representable.

\begin{proposition}[\cite{kitaev_line2,kitaev_linegraphs}] \label{prop_1}
	Let $w$ represents a graph $G = (V,E)$, and $a,b,c \in V$ and $\overline{ab}, \overline{ac} \in E$. Then
	\begin{itemize}
		\item if $\overline{bc} \not \in E$ then both the statements $\exists (abca)$ and $\exists (acba)$ are true for w, while
		\item  if $\overline{bc} \in E$ then exactly one of the statement $\forall (abca)$ and $\forall (acba)$ is true for w.
	\end{itemize}
\end{proposition}

\begin{lemma}[\cite{kneser_graphs_2020}]\label{cycle}
	Suppose that the vertices in $\{a, b, c, d\}$ induce a subgraph $S$ in a partially oriented graph such that $\overrightarrow{ab}$ and
	$\overrightarrow{bc}$ are edges, $\overline{cd}$ and $\overline{da}$ are non-oriented edges, and $S$ is different from the complete graph $K_4$. Then, the unique way to
	orient $\overline{cd}$ and $\overline{da}$ in order not to create a directed cycle or a shortcut is $\overrightarrow{ad}$ and $\overrightarrow{dc}$.
\end{lemma}

\begin{theorem}\label{non_w'}
	The line graph $L(W_5')$ is non-word-representable.
\end{theorem}

\begin{proof}
	Suppose $L(W_5')$ is word-representable and $w$ is a $k$-uniform word that represents $L(W_5')$. Since any cyclic shift of a word representing a graph also represents the same graph (cf. \cite[Proposition 5]{kitaev08}), assume that the leftmost letter of $w$ is $b_1$. Since $\{b_1,b_2,b_3,b_4\}$ induces a clique in $L(W_5')$, $w|_{\{b_1,b_2,b_3,b_4\}}$ is of the form $u^k$ (by \cite[Proposition 3]{kitaev_linegraphs}), where $u = b_1 b_{i_2} b_{i_3} b_{i_{4}}$ such that $i_2,i_3, i_4 \in \{2,3,4\}$. In the following, we show that either $u = b_1b_2b_3 b_4$ or $u = b_1 b_4 b_3 b_2$.
	
	Suppose $i_2,i_4 \in \{3,4\}$ so that $u = b_1 b_{i_2} b_2 b_{i_4}$. Applying Proposition \ref{prop_1} on the vertices $b_1, b_{i_2}$ and $b_{2}$, the statement $\forall (b_1 b_{i_2} b_{2} b_1)$ is true for $w$. Similarly, applying Proposition \ref{prop_1} on the vertices $b_1, b_{2}$ and $b_{i_4}$, the statement $\forall (b_1 b_{2} b_{i_4} b_1)$ is true for $w$.  As the vertex $e_2$ is non-adjacent to $b_{i_2}$  but adjacent to $b_1$,  the statements $\exists (b_1 e_2 b_{i_2} b_1 )$ and $\exists (b_1 b_{i_2} e_2 b_1)$ are true for $w$. Also, since the vertex $e_2$ is non-adjacent to $b_{i_4}$, the statements $\exists (b_1 e_2 b_{i_4} b_1 )$ and $\exists (b_1 b_{i_4} e_2 b_1)$ are true for $w$. From the statements $\exists (b_1 e_2 b_{i_2} b_1 )$ and $\forall(b_1 b_{i_2} b_2 b_1)$, we get $\exists (b_1 e_2 b_2 b_1)$. From the statements $\exists (b_1 b_{i_4} e_2 b_1)$ and $\forall (b_1 b_2 b_{i_4} b_1)$, we get $\exists (b_1  b_2 e_2 b_1)$. Thus, we can conclude that $\exists (b_1 e_2 b_2 b_1)$ and $\exists (b_1 b_2 e_2 b_1)$ are true for $w$. This is a contradiction to Proposition \ref{prop_1} when applied to the vertices $b_1, b_2$ and $e_2$. Thus, either $u = b_1 b_2 b_{i_3} b_{i_4}$ or $u = b_1b_{i_2} b_{i_3} b_2$. 
	
	Suppose $u = b_1 b_2 b_4 b_3$. Then applying Proposition \ref{prop_1} to the vertices $b_2,b_3$ and $b_4$,  the statement $\forall (b_2 b_4 b_3 b_2)$ is true for $w$. Similarly,  applying Proposition \ref{prop_1} to the vertices $b_2,b_3$ and $b_1$,  the statement $\forall (b_2 b_3 b_1 b_2)$ is true for $w$. Since the vertex $e_3$ is non-adjacent to $b_1$ but adjacent to $b_2$, the statements $\exists (b_2 e_3 b_{1} b_2 )$ and $\exists (b_2 b_{1} e_3 b_2)$ are true for $w$. Also, since the vertex $e_3$ is non-adjacent to $b_4$, 	the statements $\exists (b_2 e_3 b_{4} b_2 )$ and $\exists (b_2 b_{4} e_3 b_2)$ are true for $w$. From the statements   $\exists (b_2 b_1 e_{3} b_2 )$  and $\forall (b_2 b_3 b_1 b_2)$, we get  $\exists (b_2 b_3 e_{3} b_2 )$. From the statements  $\exists(b_2 e_3 b_{4} b_2 )$ and $\forall (b_2 b_4 b_3 b_2)$, we get $\exists (b_2 e_3 b_3 b_2 )$. Thus,  we can conclude that $\exists (b_2 e_3 b_3 b_2)$ and $\exists (b_2 b_3 e_3 b_2)$ are true for $w$. This is a contradiction to Proposition \ref{prop_1} when applied to the vertices $b_2, b_3$ and $e_3$. Thus,  $u = b_1 b_2 b_{3} b_{4}$.
	
	Suppose $u = b_1 b_3 b_4 b_2$. Then, applying Proposition \ref{prop_1} to the vertices $b_1,b_2$ and $b_3$,  the statement  $\forall (b_2 b_1 b_3 b_2)$ is true for $w$. Similarly, applying Proposition \ref{prop_1} to the vertices $b_2, b_3$ and $b_4$,  the statement  $\forall (b_2 b_3 b_4 b_2)$ is true for $w$.  Since the vertex $e_3$ is adjacent to $b_2$ but non-adjacent to $b_1$, the statements $\exists (b_2 e_3 b_{1} b_2 )$ and $\exists (b_2 b_{1} e_3 b_2)$ are true for $w$. Also, since $e_3$ is non-adjacent to $b_4$, the statements $\exists (b_2 e_3 b_{4} b_2 )$ and $\exists (b_2 b_{4} e_3 b_2)$ are true for $w$. From the statements $\exists (b_2 e_3 b_{1} b_2 )$ and $\forall (b_2 b_1 b_3 b_2)$, we get $\exists (b_2 e_3 b_3 b_2)$. From the statements $\exists (b_2 b_{4} e_3 b_2)$ and $\forall(b_2 b_3 b_4 b_2)$, we get $\exists (b_2 b_3 e_3 b_2)$. Thus,  we can conclude that $\exists (b_2 e_3 b_3 b_2)$ and $\exists (b_2 b_3 e_3 b_2)$ are true for $w$. This is a contradiction to Proposition \ref{prop_1} when applied to the vertices $b_2, b_3$ and $e_3$. Thus, $u = b_1 b_4 b_3 b_2$.
	
	Note that the second case, i.e.,  $u =b_1 b_4 b_3 b_2$, can be obtained from the first case, i.e., $u = b_1 b_2 b_3 b_4$, by reversing the word $w$ and taking a cyclic shift. Therefore, it is sufficient to prove the theorem for the first case.
	
	Let $D_0$ be the orientation of $L(W'_5)$ associated with the word $w$ such that for any two adjacent vertices $a$ and $b$ in $L(W'_5)$, $\overrightarrow{ab} \in D_0$ if and only if the first occurrence of $a$ occurs before the first occurrence of $b$ in $w$. Since $w$ represents $L(W'_5)$, in view of \cite[Theorem 3]{halldorsson16}, $D_0$ must be a semi-transitive orientation. However, we show that $D_0$ cannot be semi-transitive so that $L(W'_5)$ is non-word-representable.
	
	Since the vertex $b_1$ is the leftmost letter of $w$, $b_1$ must be a source in $D_0$. Thus, we have $\overrightarrow{b_1e_1}, \overrightarrow{b_1e_2}, \overrightarrow{b_1b_2}, \overrightarrow{b_1 b_3}$ and $\overrightarrow{b_1b_4}$. Further, since $w|_{\{b_1, b_2, b_3, b_4\}} = (b_1 b_2 b_3 b_4)^k$, we have $\overrightarrow{b_2 b_3}, \overrightarrow{b_2 b_4}, \overrightarrow{b_3 b_4}$ in $D_0$.
	
		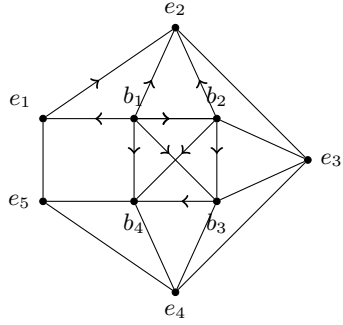
\begin{figure}[!h]
		\centering
		\begin{tikzpicture}[scale=1.1]
			
			\begin{scope}[decoration={markings,mark=at position 0.42 with {\arrow[thick]{>}}}]
				\vertex (27) at (0,0) [label=above:$b_1$] {};  
				\vertex (37) at (1,0) [label=above:$b_2 $] {}; 
				\vertex (47) at (0,-1) [label=below :$b_4$] {}; 
				\vertex (57) at (1,-1) [label=below :$b_3$] {}; 
				
				\vertex (12) at (-1.1,0) [label=above left:$e_1$] {}; 	
				\vertex (23) at (0.5,1.1) [label=above:$e_2$] {}; 
				\vertex (34) at (2.1,-0.5) 
				[label=right:$e_3 $] {}; 
				\vertex (45) at (0.5, -2.1) 
				[label=below:$e_4 $] {}; 
				\vertex (15) at (-1.1, -1) [label=left:$e_5 $] {}; 
				
				\path 
				(27) edge[postaction=decorate] (37)
				(27) edge[postaction=decorate] (37)
				(27) edge[postaction=decorate] (47)
				(27) edge[postaction=decorate] (57)
				(37) edge[postaction=decorate] (47)
				(37) edge[postaction=decorate] (57)	 
				(57) edge[postaction=decorate] (47);
				
				\path 
				(27) edge[postaction=decorate] (12)
				(27) edge[postaction=decorate] (23)
				(37) edge[postaction=decorate] (23)
				(37) edge (34)
				(57) edge (34)
				(47) edge (45)
				(57) edge (45)
				(15) edge (47);
				
				\path 
				(12) edge[postaction=decorate] (23)
				(23) edge (34)
				(34) edge (45)
				(45) edge (15)
				(12) edge (15);
			\end{scope}
		\end{tikzpicture}	
		\caption{Partial orientation of $L(W_5')$}	
		\label{fig:w5_partial}
	\end{figure}
	
	Consider the subgraph induced by $\{b_1,b_2,b_3,e_2\}$. Note that if $\overrightarrow{e_2 b_2} \in D_0$ then $\overrightarrow{b_1b_3}$ becomes the shortcutting edge of the directed path $\langle b_1,e_2,b_2,b_3 \rangle$, which is not possible as $D_0$ does not contain any shortcut. Thus,  $\overrightarrow{b_2e_2} \in D_0$. Similarly,	considering the subgraph induced by $\{b_1,b_2,e_1,e_2\}$, we must have $\overrightarrow{e_1e_2} \in D_0$. The partial orientation of $W_5'$ with respect to the semi-transitive orientation $D_0$ is shown in Fig. \ref{fig:w5_partial}.
		
	\begin{claim}
		For $2 \le i \le 4$, $\overrightarrow{e_i e_{i+1}} \in D_0$. 	
	\end{claim}
	
	Suppose $\overrightarrow{e_3e_2} \in D_0$.  Consider the subgraph induced by $\{b_1,b_2,e_2,e_3\}$. Note that if $\overrightarrow{b_2e_3} \in D_0$  then $\overrightarrow{b_1e_2}$ becomes the shortcutting edge of the directed path $\langle b_1,b_2,e_3,e_2 \rangle$, which is not possible in $D_0$ as it does not contain any shortcut. Thus, $\overrightarrow{e_3b_2} \in D_0$.
	If $\overrightarrow{b_3e_3} \in D_0$  then the vertices $\{b_2,b_3,e_3\}$ form a directed cycle, which is not possible as $D_0$ is an acyclic orientation. Thus, $\overrightarrow{e_3b_3} \in D_0$. Now applying Lemma \ref{cycle} to the vertices $\{e_3,b_3,b_4,e_4\}$, we have $\overrightarrow{e_3e_4}$ and $\overrightarrow{e_4b_4}$. If $\overrightarrow{b_3e_4} \in D_0$ then $\overrightarrow{b_1b_4}$ becomes the shortcutting edge for the directed path $\langle b_1,b_3,e_4,b_4 \rangle$, which is not possible. Thus, $\overrightarrow{e_4b_3} \in D_0$.
	Again applying Lemma \ref{cycle} to the vertices $\{e_4, b_3, b_4, e_5\}$, we have $\overrightarrow{e_4 e_5}, \overrightarrow{e_5b_4} \in D_0$. Finally, if $\overrightarrow{e_1e_5}$ then $\overrightarrow{b_1b_4}$ becomes the shortcutting edge of the directed path $\langle b_1,e_1,e_5,b_4 \rangle$. Thus, $\overrightarrow{e_5e_1} \in D_0$. But then, $\langle e_3,e_4,e_5,e_1,e_2 \rangle$ forms a shortcut with the shortcutting edge $\overrightarrow{e_3e_2}$. Hence, we have $\overrightarrow{e_2e_3} \in D_0$.

	Suppose $\overrightarrow{e_4e_3} \in D_0$. By a similar argument, considering the subgraph induced by $\{b_2,b_3,e_2,e_3\}$, we have $\overrightarrow{b_3e_3} \in D_0$. Note that if $\overrightarrow{e_3b_2} \in D_0$ then $\overrightarrow{b_1b_2}$ becomes the shortcutting edge of the directed path $\langle b_1,b_3,e_3,b_2 \rangle$, which is not possible. Thus, $\overrightarrow{b_2e_3} \in D_0$.  If $\overrightarrow{b_4e_4}$ then $\overrightarrow{b_3 e_3}$ becomes the shortcutting edge of the directed path $\langle b_3,b_4,e_4, e_3 \rangle$, which is not possible. Thus, $\overrightarrow{e_4b_4} \in D_0$. By a similar argument, considering the subgraph induced by $\{b_2,b_3,b_4,e_4\}$, we have $\overrightarrow{e_4b_3} \in D_0$.
	Applying Lemma \ref{cycle} to the vertices $\{e_4,b_3,b_4,e_5\}$, $\overrightarrow{e_4e_5}, \overrightarrow{e_5b_4} \in D_0$. Finally, if $\overrightarrow{e_5 e_1} \in D_0$ then $\overrightarrow{e_4e_3}$ becomes the shortcutting edge of the directed path $\langle e_4, e_5, e_1, e_2, e_3 \rangle$ in $D_0$, which is not possible. Otherwise, if $\overrightarrow{e_1 e_5} \in D_0$ then the edge $\overrightarrow{b_1 b_4}$ is the shortcutting edge of the directed path $\langle b_1,e_1,e_5,b_4 \rangle$, which is not possible in $D_0$. Hence, $\overrightarrow{e_3e_4} \in D_0$.
	
	Suppose $\overrightarrow{e_5e_4} \in D_0$. Then $\overrightarrow{e_1 e_5} \in D_0$; otherwise if $\overrightarrow{e_5 e_1}$ then $\overrightarrow{e_5 e_4}$ becomes the shortcutting edge of the directed path $\langle e_5, e_1,e_2,e_3,e_4\rangle$, which is not possible in $D_0$. In the subgraph induced by the vertices $\{b_1,b_4,e_1,e_5\}$, it is clear that $\overrightarrow{b_4e_5} \in D_0$, as the directed edge $\overrightarrow{e_5b_4}$  will create a shortcut $\langle b_1,e_1,e_5,b_4 \rangle$. As $\overrightarrow{b_4e_5} \in D_0$, if $\overrightarrow{b_3e_4} \in D_0$ then $\overrightarrow{b_3b_4}$ becomes the shortcutting edge of the directed path $\langle b_3, b_4, e_5, e_4 \rangle$ in $D_0$, which is not possible. If $\overrightarrow{e_4b_3} \in D_0$ then the vertices $\{b_3,b_4,e_5,e_4\}$ form a directed cycle, which is also not possible in $D_0$. Since any orientation of the edge $\overline{b_3 e_4}$ leads to a shortcut or a directed cycle, we arrive at a contradiction. Thus, $\overrightarrow{e_4e_5} \in D_0$. 
		
	As $\overrightarrow{e_i e_{i+1}} \in D_0$, for all $1 \le i \le 4$. When $\overrightarrow{e_1 e_5} \in D_0$ then $\overrightarrow{e_1e_5}$ is the shortcutting edge of the directed path $\langle e_1,e_2,e_3,e_4,e_5 \rangle$. When $\overrightarrow{e_5 e_1} \in D_0$, then the vertices $\{e_1,e_2,e_3,e_4,e_5\}$ forms a directed cycle in $D_0$. 
	Thus, $D_0$ cannot be a semi-transitive orientation  and hence $L(W_5')$ is non-word-representable.
	\qed	 
\end{proof}

\section{Mycielski Graph of Odd Cycle}

Let $G$ be a graph with vertex set $[n] = \{1, 2, \ldots, n\}$. The Mycielski graph of $G$, denoted by $\mu(G)$, is an undirected graph with the vertex set $[n] \cup \{1', 2', \ldots, n'\} \cup \{0\}$ and the edge set $E(G) \cup \{\overline{0i'}\ | \ i \in [n]\} \cup \{\overline{j i'}, \overline{i j'}\ | \ \overline{ij} \in E(G)\}$, where $E(G)$ is the edge set of $G$. For example, the Mycielski graph of an odd cycle $C_{2n+1}$ is shown in Fig. \ref{muc_2n+1}.

In \cite[Theorem 4]{kitaev_mycielski}, for $n \ge 1$, it was shown that $\mu(C_{2n+1})$ is non-word-representable. Note that $\mu(C_3)$ contains $W_5'$ as a subgraph (see Fig. \ref{muc_3}), and hence, by Theorem \ref{non_w'}, $L(\mu(C_3))$ is non-word-representable.  

\begin{remark}\label{con_2}
	Note that a graph $G$ is always an induced subgraph of its Mycielski graph. Thus, if $G$ contains $C_3$ as a subgraph then $L(\mu(G))$ is non-word-representable. 
\end{remark}

In the following, for $n \ge 2$, we show that the line graph of $\mu(C_{2n+1})$ is word-representable by giving a semi-transitive orientation to $L(\mu(C_{2n+1}))$.

\begin{figure}[h]
	\centering
	\begin{minipage}[b]{.5\textwidth}
		\centering
		\begin{tikzpicture}[scale=0.8]
			\vertex (1) at (0,0) [label= left:$1$] {};  
			\vertex (2) at (1,0) [label=above:$2$] {}; 
			\vertex (3) at (2,0) [label=above:$  $] {}; 
			\vertex (4) at (6,0) [label=right:$ $] {}; 	
			\vertex (5) at (7,0) [label=right:$2n+1$] {}; 
			\vertex (11) at (0,-1) [label=  left:$1'$] {};  
			\vertex (12) at (1,-1) [label=below:$2'$] {}; 
			\vertex (13) at (2,-1) [label=below:$ $] {}; 
			\vertex (14) at (6,-1) [label=right:$  $] {}; 	
			\vertex (15) at (7,-1) [label=right:$(2n+1)'$] {}; 
			\vertex (0) at (3.5,-2.7) [label=below:$0$] {}; 
			\node (A) at (2,-2.9) [label=below:$ $] {}; 
			
			\vertex (6) at (3,0) [label=right:$  $] {}; 	
			\vertex (7) at (4,0) [label=right:$ $] {}; 
			\vertex (8) at (5,0) [label=right:$  $] {};

			\vertex (16) at (3,-1) [label=left:$ \cdots  $] {}; 	
			\vertex (17) at (4,-1) [label=right:$ $] {}; 
			\vertex (18) at (5,-1) [label=right:$ \cdots $] {}; 	
			\path
			(1) edge (2)
			(2) edge (3)
			(3) edge[dashed](4)
			(4) edge  (5)
			(1) edge[bend left=40] (5)
			(1) edge (12)
			(1) edge (15)
			(2) edge (11)
			(2) edge (13)
			(6) edge (7)
			(7) edge (8)
			(6) edge (17)
			(17) edge (8)
			(16) edge (7)
			(7) edge (18)
			(5) edge (14)
			(5) edge (11)
			(0) edge (11)
			(0) edge (12)
			(0) edge (13)
			(0) edge (14)
			(0) edge (15)
			 
			(0) edge (16)
			(0) edge (17)
			(0) edge (18)
			;
			
		\end{tikzpicture}
		\caption{$\mu(C_{2n+1})$ }
		\label{muc_2n+1}	
	\end{minipage}%
	\begin{minipage}[b]{.55\textwidth}
		\centering
		\begin{tikzpicture}[scale=0.8]
			\vertex (1) at (0,0) [label= left:$1$] {};  
			\vertex (2) at (1,0) [label=above:$2$] {}; 
			\vertex (3) at (2,0) [label=right:$3$] {}; 
			\vertex (11) at (0,-1) [label=  left:$1'$] {};  
			\vertex (12) at (1,-1) [label=left:$2'$] {}; 
			\vertex (13) at (2,-1) [label=right:$3'$] {}; 
			\vertex (0) at (1,-2.5) [label=below:$0$] {}; 
			
			\path
			(1) edge[thick,red] (2)
			(2) edge[thick,red] (3)
			(1) edge[bend left=80,thick,red] (3)
			(1) edge[thick,red] (12)
			(1) edge[thick,red] (13)
			(2) edge (11)
			(2) edge[thick,red] (13)
			(3) edge[thick,red] (12)
			(3) edge (11)
			(0) edge (11)
			(0) edge[thick,red] (12)
			(0) edge[thick,red] (13)
			
			;
			
		\end{tikzpicture}
		\caption{$\mu(C_{3})$ }
		\label{muc_3}	
	\end{minipage}%
\end{figure}

To describe the line graph $L(\mu(C_{2n+1}))$, we give a double-indexed label for each edge of $\mu(C_{2n+1})$. For $1 \le i < j \le 2n+1$, let $c_{ij}$ denote the edge of the cycle $C_{2n+1}$ in $\mu(C_{2n+1})$ with endpoints $i$ and $j$. For $0 \le i \le 2n+1$ and $1 \le j \le (2n+1)$, let $a_{ij'}$ be the edge of $\mu(C_{2n+1})$ with end points $i$ and $j'$. Note that two vertices in $L({\mu(C_{2n+1})})$ are adjacent if and only if one of their indices are the same.

We consider the following two subgraphs of $\mu(C_{2n+1})$ which partition its edge set, as per the labeling $c_{ij}$'s and $a_{ij'}$'s given to the edges. The part with the edges $c_{ij}$'s is the original graph $C_{2n+1}$ and let $B_{2n+1}$ be the spanning subgraph with the remaining edges $a_{ij'}$'s. Note that $B_{2n+1}$ is a bipartite graph. 

To prove our main result stated in Theorem \ref{mu_cn}, we give an orientation $D$ to the edges of $L(\mu(C_{2n+1}))$, for $n \ge 2$, and show that $D$ is a semi-transitive orientation. The orientation $D$ is divided into three parts as described in the following steps:

\begin{enumerate}[label=\Roman*.]
	\item \textbf{Semi-transitive orientation assigned to $L(B_{2n+1})$}: The graph $B_{2n+1}$ is a subgraph of the complete bipartite graph $K_{2n+2, 2n+1}$, so that $L(B_{2n+1})$ is an induced subgraph of $L(K_{2n+2, 2n+1})$. Note that the line graph of a complete bipartite graph $L(K_{m,n})$ is isomporphic to the Cartesian product of the complete graphs $K_m$ and $K_n$  (cf. \cite{werra_1999}) as shown  Fig. \ref{kmn}. 
	The vertices in each rectangular box form a clique, i.e., for $1 \le i \le m$, $L_i = \{a_{ij'} |\  1 \le j\le n\}$ and, for $1 \le j \le n$, $L_j' = \{a_{ij'} |\ 1 \le i\le m \}$ form cliques in $L(K_{m,n})$. 
	From \cite[Theorem 5.4.10]{kitaev15mono}, the line graph $L(K_{m,n})$ is word-representable. By hereditary property of word-representable graphs, $L(B_{2n+1})$ is word-representable, and hence admits a semi-transitive orientation. In view of \cite[Theorem 5.4.10]{kitaev15mono}, consider the following semi-transitive orientation of $L(K_{2n+2,2n+1})$ restricted to the graph $L(B_{2n+1})$, as shown in Fig. \ref{b_n}.
	\begin{enumerate}[label=(\roman*)]
		\item For $0 \le i \le 2n+1$, $\overrightarrow{a_{ij_1'}a_{ij_2'}}$ if $j_1 < j_2$. 
		\item For $1 \le j \le 2n+1$, $\overrightarrow{a_{i_1j'}a_{i_2j'}}$ if $i_1 > i_2$. 
	\end{enumerate}
	
	\begin{figure}
		\centering
		\begin{minipage}{.4\textwidth}			
			\begin{tikzpicture}[scale=1]
				\vertex (1) at (0,0) [label=above:$1$] {};  
				\vertex (2) at (1,0) [label=above:$2$] {}; 
				\vertex (3) at (2,0) [label=above:$ $] {}; 
				\vertex (4) at (3,0) [label=right:$ $] {}; 	
				\vertex (5) at (4,0) [label=above:$m$] {}; 
				\vertex (11) at (0,-1) [label=below:$1'$] {};  
				\vertex (12) at (1,-1) [label=below:$2'$] {}; 
				\vertex (13) at (2,-1) [label=below:$ $] {}; 
				\vertex (14) at (3,-1) [label=right:$ $] {}; 	
				\vertex (15) at (4,-1) [label=below:$n'$] {}; 
				\node (A) at (2.7,-2) [label=left:$K_{m,n}$] {}; 
				\path
				(1) edge (11)
				(1) edge (12)
				(1) edge (13)
				(1) edge (14)
				(1) edge (15)
				(2) edge (11)
				(2) edge (12)
				(2) edge (13)
				(2) edge (14)
				(2) edge (15)
				(3) edge (11)
				(3) edge (12)
				(3) edge (13)
				(3) edge (14)
				(3) edge (15)
				(4) edge (11)
				(4) edge (12)
				(4) edge (13)
				(4) edge (14)
				(4) edge (15)
				(5) edge (11)
				(5) edge (12)
				(5) edge (13)
				(5) edge (14)
				(5) edge (15);
			
			\end{tikzpicture}	
		\end{minipage}
		\begin{minipage}{.5\textwidth}
			\centering
			\begin{tikzpicture}[scale=0.8]
				\begin{scope}[decoration={markings,mark=at position 0.65 with {\arrow[thick]{>}}}]
					\vertex (11) at (0,0) [label=above:$a_{m1'}$] {};  
					\vertex (12) at (1,0) [label=above:$a_{m2'}$] {}; 
					\vertex (13) at (2,0) [label=above:$ $] {}; 
					\vertex (14) at (3,0) [label=right:$ $] {}; 	
					\vertex (15) at (4,0) [label=above:$a_{mn'}$] {}; 
					\vertex (21) at (0,-1) [label=above:$ $] {};  
					\vertex (22) at (1,-1) [label=above:$ $] {}; 
					\vertex (23) at (2,-1) [label=above:$ $] {}; 
					\vertex (24) at (3,-1) [label=:$ $] {}; 	
					\vertex (25) at (4,-1) [label=below:$ $] {};
					\vertex (31) at (0,-2) [label=above:$ $] {};  
					\vertex (32) at (1,-2) [label=above:$ $] {}; 
					\vertex (33) at (2,-2) [label=above:$ $] {}; 
					\vertex (34) at (3,-2) [label=right:$ $] {}; 	
					\vertex (35) at (4,-2) [label=above:$ $] {}; 
					\vertex (41) at (0,-3) [label=below:$a_{21'}$] {};  
					\vertex (42) at (1,-3) [label=below:$a_{22'}$] {}; 
					\vertex (43) at (2,-3) [label=below:$ $] {}; 
					\vertex (44) at (3,-3) [label=right:$ $] {}; 	
					\vertex (45) at (4,-3) [label=below:$a_{2n'}$] {};  
					\vertex (51) at (0,-4) [label=below:$a_{11'}$] {};  
					\vertex (52) at (1,-4) [label=below:$a_{12'}$] {}; 
					\vertex (53) at (2,-4) [label=below:$ $] {}; 
					\vertex (54) at (3,-4) [label=right:$ $] {}; 	
					\vertex (55) at (4,-4) [label=below:$a_{1n'}$] {}; 
					
					\draw[red] (-1,-0.25) rectangle (5 ,0.6) node[pos=.5] {$ $};
					\node (L5) at (5,0.2) [label=right:$L_m$ ] {}; 
					\draw[red] (-1,-0.8) rectangle (5 ,-1.2) node[pos=.5] {$ $};
					\node (L4) at (5,-1) [label=right:$ $ ] {};
					\draw[red ] (-1,-1.8) rectangle (5 ,-2.2) node[pos=.5] {$ $};
					\draw[red ] (-1,-2.8) rectangle (5 ,-3.5) node[pos=.5] {$ $};
					\node (L2) at (5,-3.2) [label=right:$L_2$ ] {};
					\draw[red ] (-1,-3.8) rectangle (5 ,-4.5) node[pos=.5] {$ $};
					\node (L2) at (5,-4.2) [label=right:$L_1$ ] {};
					
					\draw[blue] (-0.45,-5) rectangle (0.35 ,1) node[pos=.5] {$ $};
					\node (L1') at (0.06,-5) [label=below:$L_1'$ ] {};
					\draw[blue] (0.6,-5) rectangle (1.35 ,1) node[pos=.5] {$ $};
					\node (L2') at (1,-5) [label=below:$L_2'$ ] {};
					\draw[blue] (1.8,-5) rectangle (2.2 ,1) node[pos=.5] {$ $};
					\draw[blue] (2.8,-5) rectangle (3.2 ,1) node[pos=.5] {$ $};
					\draw[blue] (3.55,-5) rectangle (4.35 ,1) node[pos=.5] {$ $};
					\node (Ln') at (4.05,-5) [label=below:$L_n'$ ] {};

					\path
					(11) edge[postaction={decorate}] (12)
					(12) edge[postaction={decorate}] (13)	
					(13) edge[postaction={decorate}] (14)
					(14) edge[postaction={decorate}] (15)
					(21) edge[postaction={decorate}] (22)
					(22) edge[postaction={decorate}] (23)	
					(23) edge[postaction={decorate}] (24)
					(24) edge[postaction={decorate}] (25)
					(31) edge[postaction={decorate}] (32)
					(32) edge[postaction={decorate}] (33)	
					(33) edge[postaction={decorate}] (34)
					(34) edge[postaction={decorate}] (35)
					(41) edge[postaction={decorate}] (42)
					(42) edge[postaction={decorate}] (43)	
					(43) edge[postaction={decorate}] (44)
					(44) edge[postaction={decorate}] (45)
					(51) edge[postaction={decorate}] (52)
					(52) edge[postaction={decorate}] (53)	
					(53) edge[postaction={decorate}] (54)
					(54) edge[postaction={decorate}] (55)
					(11) edge[postaction={decorate}] (21)
					(21) edge[postaction={decorate}] (31)
					(31) edge[postaction={decorate}] (41)	
					(41) edge[postaction={decorate}] (51)
					(12) edge[postaction={decorate}] (22)
					(22) edge[postaction={decorate}] (32)
					(32) edge[postaction={decorate}] (42)	
					(42) edge[postaction={decorate}] (52)
					(13) edge[postaction={decorate}] (23)
					(23) edge[postaction={decorate}] (33)
					(33) edge[postaction={decorate}] (43)	
					(43) edge[postaction={decorate}] (53)
					(14) edge[postaction={decorate}] (24)
					(24) edge[postaction={decorate}] (34)
					(34) edge[postaction={decorate}] (44)	
					(44) edge[postaction={decorate}] (54)
					(15) edge[postaction={decorate}] (25)
					(25) edge[postaction={decorate}] (35)
					(35) edge[postaction={decorate}] (45)	
					(45) edge[postaction={decorate}] (55);
				\end{scope}
			\end{tikzpicture}
			
			$L(K_{m,n}) \cong K_m \square K_n$
			
		\end{minipage}		
		\caption{$K_{m,n}$ and $L(K_{m,n})$ with a semi-transitive orientation}
		\label{kmn}	
	\end{figure}
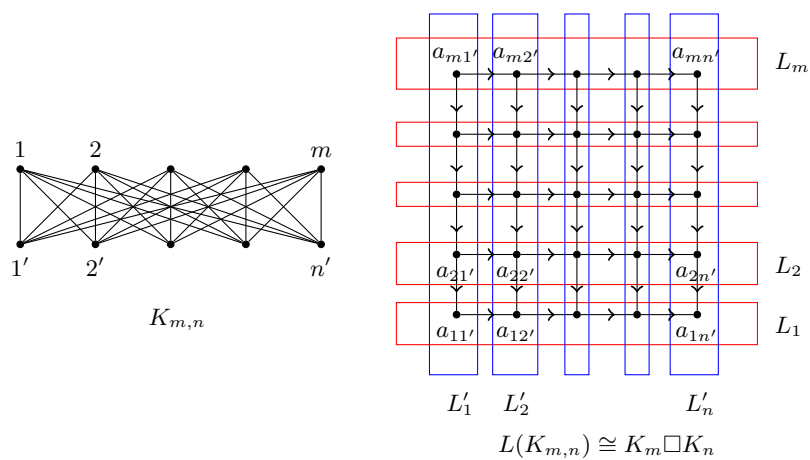
	
		\begin{figure}[h]
			
		\centering
		\begin{minipage}[b]{.5\textwidth}
		 
			\begin{tikzpicture}[scale=0.9 ]
				\begin{scope}[decoration={markings,mark=at position 0.6 with {\arrow[thick]{>}}}]
					\vertex (11) at (0,0) [label=above:$a_{(2n+1)1'}$] {};  
					\vertex (14) at (3,0) [label=above:$a_{(2n+1)(2n)'}$] {};			 
					
					\vertex (23) at (2,-1) [label=above:$a_{(2n)(2n-1)'}$] {}; 
					
					\vertex (25) at (4,-1) [label=above:$a_{(2n)(2n+1)'}$] {};
					
					\vertex (32) at (1,-1.5) [label=above:$ $] {}; 
					
					\vertex (34) at (3,-1.5) [label=right:$ $] {}; 	
					
					\vertex (41) at (0,-3) [label=left:$a_{21'}$] {}; 
					
					\vertex (412) at (1.3,-3) [label=right:$  $] {}; 
					
					\vertex (43) at (2,-2.3) [label=right:$  $] {}; 
					
					\vertex (413) at (1.3,-2.3) [label=right:$  $] {}; 
					
					\vertex (52) at (1,-4) [label=left:$a_{12'}$] {}; 
					
					\vertex (55) at (4,-4) [label=right:$a_{1{(2n+1)}'}$] {}; 
					
					\vertex (61) at (0,-5) [label=left:$a_{01'}$] {};  
					\vertex (62) at (1,-5) [label=above left:$a_{02'}$] {}; 
					\vertex (63) at (2,-5) [label=above left:$\cdots$] {}; 
					\vertex (64) at (3,-5) [label=below:$ $] {}; 	
					\vertex (65) at (4,-5) [label=right:$a_{0(2n+1)'}$] {};

					\node (L0) at (5.4,-5) [label=right:$L_0$ ] {}; 
					
					\node (L5) at (5.4,0 ) [label=right:$L_{2n+1}$ ] {}; 
					\node (L4) at (5.4,-1) [label=right:$ L_{2n}$ ] {};
					\node (L2) at (5.4,-3 ) [label=right:$L_2$ ] {};
					\node (L2) at (5.4,-4) [label=right:$L_1$ ] {};
					
					\node (L1') at (0.06,-5.5) [label=below:$L_1'$ ] {};
					\node (L2') at (1.06,-5.5) [label=below:$L_2'$ ] {};
					\node (Ln') at (4.06,-5.5) [label=below:$L_{2n+1}'$ ] {};

					\path
					(11) edge[postaction={decorate}] (14)
					
					(23) edge[postaction={decorate}] (25)

					(32) edge[postaction={decorate}] (34)

					(41) edge[postaction={decorate}] (412)
					
					(23) edge[postaction={decorate}] (43)
					(413) edge[postaction={decorate},dashed] (412)
						(413) edge[postaction={decorate},dashed] (43)
					
					(52) edge[postaction={decorate}] (55)

					(14) edge[postaction={decorate}] (34) 
					(32) edge[postaction={decorate}] (52)
					(11) edge[postaction={decorate}]  (41)
					(25) edge[postaction={decorate}] (55)
					(61)  edge[postaction={decorate}] (62)
					(62)  edge[postaction={decorate}]  (63)
					(63) edge[postaction={decorate}] (64)
					(64)  edge[postaction={decorate}] (65)
					(41)  edge[postaction={decorate}] (61)
					(52)  edge[postaction={decorate}] (62)
					(43)  edge[postaction={decorate}] (63)
					(34)  edge[postaction={decorate}] (64)
					(55)  edge[postaction={decorate}] (65)
					(11)  edge[postaction={decorate},bend right=10] (61)
					(32)  edge[postaction={decorate},bend right=10] (62)
					(23)  edge[postaction={decorate},bend right=10] (63)
					(14)  edge[postaction={decorate},bend right=10] (64)
					(25)  edge[postaction={decorate},bend right=10] (65)
					(61)  edge[postaction={decorate},bend right=20] (63)
					(61)  edge[postaction={decorate},bend right=20] (64)
					(61)  edge[postaction={decorate},bend right=20] (65)
					(62)  edge[postaction={decorate},bend right=20] (64)
					(62)  edge[postaction={decorate},bend right=20] (65)
					(63)  edge[postaction={decorate},bend right=20] (65)
					;
				\end{scope}
				
			\end{tikzpicture}
				\caption{$L(B_{2n+1})$}
			\label{b_n} 
		\end{minipage}
		\begin{minipage}[b]{.45\textwidth}
			\centering
		\qquad \qquad \qquad\begin{tikzpicture}[scale=.7]	
			\begin{scope}[decoration={markings,mark=at position 0.6  with {\arrow[thick]{<}}}]		
				\vertex (012) at (6,-3.5) [label=below:$c_{12}$] {};
				\vertex (023) at (6,-2.5) [label=left:$c_{23}$] {};
				\vertex (034) at (6,-1.5) [label=left:$c_{34} $] {};
				\vertex (045) at (6,-0.5) [label=right:$  $] {};
				\vertex (056) at (6,0.5) [label=left:$  $] {};
				\vertex (067) at (6,1.5) [label=above:$c_{(2n)(2n+1)}  $] {};
				\vertex (015) at (7.5,-1) [label=right:$c_{1(2n+1)} $] {};
				\path
				(012) edge[bend right =1,postaction={decorate}](023)
				(023) edge[bend right =1,postaction={decorate}] (034)
				(034) edge[bend right=1,postaction={decorate},dashed] (045)
				(045) edge[bend left=1,postaction={decorate},dashed] (056)
				(056) edge[bend right =1,postaction={decorate}] (067)
				(015) edge[bend right =1,postaction={decorate}] (067)
				(012) edge[bend right =1,postaction={decorate}] (015);
			\end{scope}
		\end{tikzpicture}
		\caption{$L(C_{2n+1})$}
		\label{lc_2n+1}
		\end{minipage}
	\end{figure}

	\item \textbf{Semi-transitive orientation assigned to $L(C_{2n+1})$}: 
	Consider the following orientation on $L(C_{2n+1})$, as shown in Fig. \ref{lc_2n+1}.
	\begin{enumerate}[label=(\roman*)]
		\item For $1\le i \le 2n-1$,  $\overrightarrow{c_{(i+1)(i+2)} c_{i(i+1)} }$.
		\item $\overrightarrow{c_{1(2n+1)} c_{12}}$ and $ \overrightarrow{c_{(2n)(2n+1)} c_{1{(2n+1)}}}$.
	\end{enumerate}
	Note that this orientation does not contain any directed cycles and shortcuts and hence a semi-transitive orientation.

	\item \textbf{Orientation of the edges between $L(B_{2n+1})$ and $L(C_{2n+1})$}: Each edge $\overline{ac}$ in $L(\mu(C_{2n+1}))$, where $a$ is a vertex in $L(B_{2n+1})$ and $c$ is a vertex in  $L(C_{2n+1})$, orient  $\overrightarrow{ac}$ in $D$.
\end{enumerate}

An example of orientation $D$ assigned to $L(\mu(C_{5}))$ is depicted in Fig. \ref{l_mu_c5}.
	
\begin{remark}\label{level}
			
	In the subgraph $L(B_{2n+1})$, the vertices in each $L_i$ are as follows.
		\begin{align*}
				L_0 &= \{a_{01'}, a_{02'}, \ldots,a_{0(2n+1)'} \},\\
				L_1 &= \{a_{12'}, a_{1(2n+1)'}\}, \\
			\forall \ 2 \le i \le 2n,\quad L_i &= \{a_{i(i-1)'},  a_{i(i+1)'}\} \text{ and}\\
				L_{2n+1} &= \{a_{(2n+1)1'}, a_{(2n+1)(2n)'}\}.
		\end{align*}
		Further, the vertices in each $L_j'$ are as follows.
		\begin{align*}
			L_1' &= \{a_{01'}, a_{21'}, a_{(2n+1)1'}\}, \\
			\forall \ 2 \le j \le 2n,\quad L_j' &= \{a_{0j'},a_{(j-1)j'},  {a_{(j+1)j'}}\} \text{ and}\\
			L_{2n+1}' &= \{a_{0(2n+1)'}, a_{1(2n+1)'}, a_{(2n)(2n+1)'}\}.
		\end{align*}
	\end{remark}
	
	\begin{remark}\label{rem:step1} 
		As per Step I, the orientation $D$ restricted  to the induced subgraph  $L(B_{2n+1})$ (see Fig. \ref{b_n}), we have the following observations.
		\begin{enumerate}[label=\roman*)]
			\item There  is no directed path from any vertex of $L_{i}$ to any vertex of $L_{j}$ when $j > i$ in $L(B_{2n+1})$.
			\item For all $ 2 \le i \le 2n+1$, there is no directed path from a vertex of $L_i$ to  any vertex of $L_{i-1}$ in $L(B_{2n+1})$. 
			\item There  is no directed path from any vertex of $L_{i}'$ to any vertex of $L_{j}'$ when $j < i$ in $L(B_{2n+1})$.
			\item For all $ 1 \le i \le 2n$, there is no directed path from a vertex of $L_i'$ to  any vertex of $L_{i+1}'$ in $L(B_{2n+1})$. 
		\end{enumerate}
	\end{remark}

	\begin{remark}\label{path}
		As per Step II, there is a directed path from the  vertex $c_{i(i+1)}$, for $2 \le i\le 2n$, to a vertex $c_{j(j+1)}$ in $D$ if $j <i$. Accordingly, there is a directed path from $a \in L_j$ to $c_{i(i+1)}$ if  $j \ge i$. Note that if $j \in \{1, 2n, 2n+1\}$ then there is a directed path from $a \in L_j$ to $c_{1(2n+1)}$.	\end{remark}
	
		\begin{remark}\label{rem:step3}
		As  per Step III, there is no directed edge from a vertex of $L(C_{2n+1})$ to a vertex of $L(B_{2n+1})$ in the orientation $D$ of  $L(\mu(C_{2n+1}))$. Thus, there cannot be any directed path from a vertex of $L(C_{2n+1})$ to a vertex of $L(B_{2n+1})$. Consequently, Remark \ref{rem:step1} is true for the orientation $D$ assigned to $L(\mu(C_{2n+1}))$.		
	\end{remark}

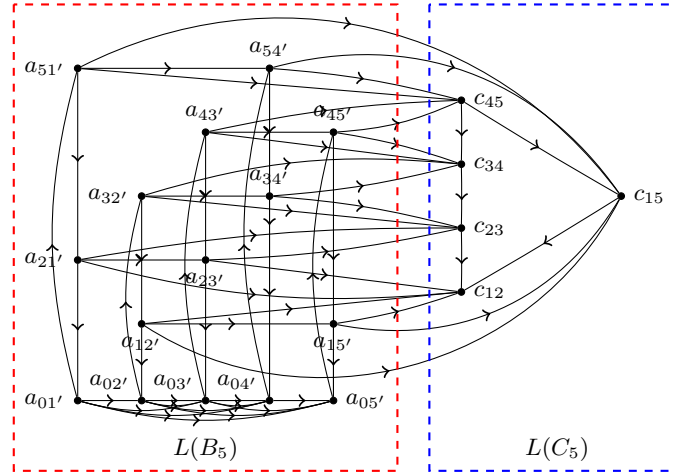
\begin{figure}[h]
\centering
	 
		\begin{tikzpicture}[scale=0.85]
			\begin{scope}[decoration={markings,mark=at position 0.55  with {\arrow[thick]{<}}}]
				\vertex (11) at (0,0) [label=left:$a_{51'}$] {};  
				
				\vertex (14) at (3,0) [label=above:$a_{54'} $] {};

				\vertex (23) at (2,-1) [label=above:$a_{43'} $] {}; 
				
				\vertex (25) at (4,-1) [label=above:$a_{45'}$] {};
				
				\vertex (32) at (1,-2) [label=left:$a_{32'} $] {}; 
				
				\vertex (34) at (3,-2) [label=above:$a_{34'} $] {}; 	
				
				\vertex (41) at (0,-3) [label=left:$a_{21'}$] {};  
				
				\vertex (43) at (2,-3) [label=below:$a_{23'} $] {}; 
				
				\vertex (52) at (1,-4) [label=below:$a_{12'}$] {}; 
				
				\vertex (55) at (4,-4) [label=below:$a_{15'}$] {}; 
				
				\vertex (01) at (0,-5.2) [label=left:$a_{01'}$] {};
				\vertex (02) at (1,-5.2) [label=above left:$a_{02'}$] {};
				\vertex (03) at (2,-5.2) [label=above left:$a_{03'} $] {};
				\vertex (04) at (3,-5.2) [label=above left:$a_{04'}$] {};
				\vertex (05) at (4,-5.2) [label=right:$a_{05'} $] {};

				\vertex (012) at (6,-3.5) [label=right:$c_{12}$] {};
				\vertex (023) at (6,-2.5) [label=right:$c_{23}$] {};
				\vertex (034) at (6,-1.5) [label=right:$c_{34} $] {};
				\vertex (045) at (6,-0.5) [label=right:$c_{45} $] {};
				\vertex (015) at (8.5,-2) [label=right:$c_{15} $] {};
				\node (B) at (2,-5.5) [label=below:$L(B_5) $] {};
				\node (B) at (7.5,-5.5) [label=below:$L(C_5) $] {};
				
				\path
				
				(012) edge[bend left=5,postaction={decorate}] (55)
				(012) edge[postaction={decorate}] (52)
				(012) edge[bend left=10,postaction={decorate}] (41)
				(012) edge[postaction={decorate}] (43)
				(023) edge[bend left=5,postaction={decorate}] (43)
				(023) edge[bend right=5,postaction={decorate}] (41)
				(023) edge[bend right=5,postaction={decorate}] (34)
				(023) edge[postaction={decorate}] (32)
				(034) edge[bend left=5,postaction={decorate}] (34)
				(034) edge [bend right=11,postaction={decorate}] (32)
				(034) edge[bend right=5,postaction={decorate}] (25)
				(034) edge[postaction={decorate}] (23)
				(045) edge[bend left=10,postaction={decorate}] (25)
				(045) edge[bend right=5,postaction={decorate}] (23)
				(045) edge[bend right=5,postaction={decorate}] (14)
				(045) edge[postaction={decorate}]  (11)
				(015) edge[bend right=40,postaction={decorate}] (11)
				(015) edge[bend right=35,postaction={decorate}] (14)
				(015) edge[bend left=48,postaction={decorate}] (52)
				(015) edge[bend left=35,postaction={decorate}] (55)
				;
				\path
				(012) edge[bend right =1,postaction={decorate}](023)
				(023) edge[bend right =1,postaction={decorate}] (034)
				(034) edge[bend right=1,postaction={decorate}] (045)
				(015) edge[bend left=1,postaction={decorate}] (045)
				(012) edge[bend right =1,postaction={decorate}] (015);

				\path 
				(55) edge[postaction={decorate}] (52)
				(43) edge[postaction={decorate}] (41)
				(34) edge[postaction={decorate}] (32)
				(25) edge[postaction={decorate}] (23)
				(14) edge[postaction={decorate}] (11) 
				(41) edge[postaction={decorate}] (11)
				(52) edge[postaction={decorate}] (32)
				(43) edge[postaction={decorate}] (23)
				(34) edge[postaction={decorate}] (14)
				
				(01) edge[postaction={decorate}] (41)
				
				(02) edge[postaction={decorate}] (52)
				
				(03) edge[postaction={decorate}] (43)
				
				(04) edge[postaction={decorate}] (14)
				
				(04) edge[postaction={decorate}] (34)
				
				(05) edge[postaction={decorate}] (25)
				
				(05) edge[postaction={decorate}] (55)
				
				;
				\path
				(02) edge[ postaction={decorate}] (01)
			 	(03) edge[bend left = 15,postaction={decorate}] (01)
			 	(04) edge[bend left=15,postaction={decorate}] (01)
			 	(05) edge[bend left=15,postaction={decorate}] (01)

			 	(03) edge[bend left=15,postaction={decorate}] (02)
			 	(04) edge[bend left=15,postaction={decorate}] (02)
			 	(05) edge[bend left=15,postaction={decorate}] (02)

			 	(04) edge[bend left=15,postaction={decorate}] (03)
			 	(05) edge[bend left=15,postaction={decorate}] (03)
			 	
				(03) edge[postaction={decorate}] (02)
				(04) edge[postaction={decorate}] (03)
				(05) edge[postaction={decorate}] (04);
				
				\draw[dashed,thick, red] (-1,1) rectangle (5 ,-6.3) node[pos=.5] {$ $};
				\draw[dashed,thick, blue] (5.5,1) rectangle (9.5 ,-6.3) node[pos=.5] {$ $};
				\path 
				(11) edge[ bend right=15,postaction={decorate}] (01)
			
				(32) edge[bend right=15,postaction={decorate}] (02)

				(23) edge[bend right=15,postaction={decorate}] (03)
				
				(14) edge[bend right=15,postaction={decorate}] (04)
				
				(25) edge[bend right=15,postaction={decorate}] (05);
				
			\end{scope}
		\end{tikzpicture}
		\caption{$L(\mu(C_5))$ with the semi-transitive oreintation $D$}
		\label{l_mu_c5}	
 
\end{figure}

To conclude that $D$ is semi-transitive,  we need the following two lemmas.
 
\begin{lemma}\label{acyclic}
	For $n \ge 2$, the orientation $D$ of $L(\mu(C_{2n+1}))$ is acyclic.
\end{lemma}
\begin{proof}
	Suppose the orientation $D$  of $L(\mu(C_{2n+1}))$ contains a directed cycle, say $\langle b_1, b_2, \ldots, b_k, b_1 \rangle$. Since the orientation $D$ restricted to the induced subgraphs  $L(B_{2n+1})$ and $L(C_{2n+1})$  of $L(\mu(C_{2n+1}))$ are semi-transtive, the directed cycle $\langle b_1, b_2, \ldots, b_k, b_1 \rangle$ must involve some vertices of $L(B_{2n+1})$ and $L(C_{2n+1})$. 
	
	Suppose $b_1$ is a vertex of $L(C_{2n+1})$. Then there exists $b_j$ ($ 2 \le j \le k$) such that $b_j$ is a vertex of $L(B_{2n+1})$. This implies that there is a directed path from $b_1$ to $b_j$, contradicting Remark \ref{rem:step3}.
	
	Suppose $b_1$ is a vertex of $L(B_{2n+1})$. Then there exists $b_j$ ($2 \le j \le k$) such that  $b_j$ is a vertex of $L(C_{2n+1})$. Clearly, there is a directed path from $b_j$ to $b_1$, contradicting Remark \ref{rem:step3}.
	
	Hence, the orientation $D$ is acyclic. \qed
\end{proof}

\begin{lemma}\label{semi}
	For $n \ge 2$, the orientation $D$ of $L(\mu(C_{2n+1}))$ does not contain any shortcut.
\end{lemma}

\begin{proof}
	Suppose that the orientation $D$ contains a shortcut, say  $X = \langle b_1,b_2,\ldots,b_k \rangle$, for $k\ge 4$, with the shortcutting edge $\overrightarrow{b_1  b_k}$. Since the orientation $D$ restricted to $L(B_{2n+1})$ and $L(C_{2n+1})$ are semi-transitive, the shortcut $X$ must involve some vertices of $L(B_{2n+1})$ and some vertices of $L(C_{2n+1})$. However, we show that none of the vertices of $L(C_{2n+1})$ are in $X$. First we prove the following claim.
	
	\begin{claim}
		There exists $2\le i\le k$ such that $b_1, b_2, \ldots, b_{i-1}$ are vertices of $L(B_{2n+1})$ and $b_{i}, \ldots, b_k$ are vertices of $L(C_{2n+1})$. 
	\end{claim} 
	
	\begin{tabular}{rr}
		\hspace{2cm} & \begin{minipage}[r]{.75\textwidth}
			\textit{Proof of the Claim.}
				Suppose $b_r$ is a vertex of $L(C_{2n+1})$ and $b_s$ is a vertex of $L(B_{2n+1})$ such that $1 \le r<s\le k$. Then there is a directed path from $b_r$ to $b_s$, which is a contradiction to Remark \ref{rem:step3}. Thus, if $b_i$ is a vertex of $L(B_{2n+1})$ then $b_j$ is a vertex of $L(B_{2n+1})$ for $j < i$. Since the shortcut $X$ contains both vertices of $L(B_{2n+1})$ and $L(C_{2n+1})$, there exists  $2\le i\le k$ such that $b_1, b_2, \ldots, b_{i-1}$ are vertices of $L(B_{2n+1})$ and $b_{i}, \ldots, b_k$ are vertices of $L(C_{2n+1})$. 
		\end{minipage}
	\end{tabular}
	
	As a consequence, we see that $b_1$ and $b_k$ must be a vertex of $L(B_{2n+1})$ and $L(C_{2n+1})$, respectively. Through the following cases, we prove that none of the vertices of $L(C_{2n+1})$ is $b_k$.	
	
	\begin{itemize}
		\item[-] Case 1: $b_k = c_{12}$.  Since $b_1$ is a vertex of $L(B_{2n+1})$ and adjacent to $c_{12}$, the vertex $b_1$ must be among the vertices $a_{12'}, a_{1(2n+1)'}, a_{21'}$ and $a_{23'}$. Also, since  $b_{k-1}$ is adjacent to $c_{12}$, the vertex $b_{k-1}$ must be among the vertices $c_{1(2n+1)}, c_{23}, a_{12'}, a_{1(2n+1)'}, a_{21'}$ and $a_{23'}$. We handle this in the following four subcases.
		\begin{itemize}
			\item Subcase 1.1: $b_1 = a_{12'}$. By Remark \ref{rem:step3} and Remark \ref{path}, there is no directed path from $a_{12'}$ to the vertices $a_{21'}, a_{23'}$ and $c_{23}$. Thus, the only possible shortcut is $\langle a_{12'}, a_{1(2n+1)'}, c_{1(2n+1)}, c_{12} \rangle$. However, these vertices induce a clique (see Fig. \ref{ind_sub_1}) in $L(\mu(C_{2n+1}))$, and hence  $b_1$ cannot be $a_{12'}$.
			
			\begin{figure}[h!]
				\centering
				\begin{tikzpicture}[scale=0.8]
					\begin{scope}[decoration={markings,mark=at position 0.5 with {\arrow[thick]{>}}}]
						\vertex (1) at (0,0) [label=left:$\scriptstyle{a_{12'}}$] {};
						\vertex (2) at (2,0) [label=below:$\scriptstyle{a_{1(2n+1)'}}$] {}; 
						\vertex (3) at (4,0) [label=below:$\scriptstyle{c_{1(2n+1)}}$] {};
						\vertex (4) at (6,0) [label=right:$\scriptstyle{c_{12}}$] {};

						\path
						(1) edge[postaction={decorate}] (2)
						(2) edge[postaction={decorate}] (3)
						(3) edge[postaction={decorate}] (4)
						(1) edge[bend right = 30, postaction={decorate}] (3)
						(1) edge[bend left = 20,postaction={decorate}] (4)
						(2) edge[bend right = 30, postaction={decorate}] (4);
					\end{scope}
				\end{tikzpicture}
				\caption{The subgraph of $L(\mu(C_{2n+1}))$ induced by $a_{12'}, a_{1(2n+1)'}, c_{1(2n+1)}$ and $c_{12}$}
				\label{ind_sub_1}			
			\end{figure}
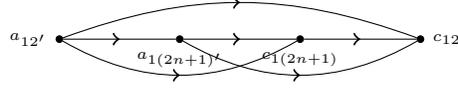
			\item Subcase 1.2: $b_1 = a_{1(2n+1)'}$. By Remark \ref{rem:step3} and Remark \ref{path}, there is no directed path from $a_{1(2n+1)'}$ to $c_{23}, a_{12'}, a_{21'}$ and $a_{23'}$. This implies that $b_k$ must be $c_{1(2n+1)}$. Clearly, there is no directed path of length at least two from $a_{1(2n+1)'}$ to $c_{1(2n+1)}$. The shortcut $X$ must be a clique of size three, which is a contradiction to the length of a shortcut. Thus, $b_1$ cannot be $a_{1(2n+1)'}$.
			\item Subcase 1.3: $b_1 = a_{21'}$. By Remark \ref{rem:step3} and Remark \ref{path}, there is no directed path from $a_{21'}$ to $a_{12'}$, $a_{1(2n+1)'}$ and $c_{1(2n+1)}$. Thus, the only possible shortcut is $\langle a_{21'}, a_{23'}, c_{23}, c_{12} \rangle$. However, these vertices induces a clique in $L(\mu(C_{2n+1}))$ (see Fig. \ref{ind_sub_12}). Hence, $b_1$ cannot be $a_{21'}$.
						
			\begin{figure}[h!]
				\centering
				\begin{tikzpicture}[scale=0.8]
					\begin{scope}[decoration={markings,mark=at position 0.5 with {\arrow[thick]{>}}}]
						\vertex (1) at (0,0) [label=left:$\scriptstyle{a_{21'}}$] {};
						\vertex (2) at (2,0) [label=below:$\scriptstyle{a_{23'}}$] {}; 
						\vertex (3) at (4,0) [label=below:$\scriptstyle{c_{23}}$] {};
						\vertex (4) at (6,0) [label=right:$\scriptstyle{c_{12}}$] {};

						\path
						(1) edge[postaction={decorate}] (2)
						(2) edge[postaction={decorate}] (3)
						(3) edge[postaction={decorate}] (4)
						(1) edge[bend right = 30, postaction={decorate}] (3)
						(1) edge[bend left = 20,postaction={decorate}] (4)
						(2) edge[bend right = 30, postaction={decorate}] (4);
					\end{scope}
				\end{tikzpicture}
				\caption{The subgraph of $L(\mu(C_{2n+1}))$ induced by $a_{21'}, a_{23'}, c_{23}$ and $c_{12}$}
				\label{ind_sub_12}			
			\end{figure}
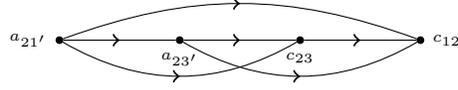
			
			\item Subcase 1.4:  $b_1 = a_{23'}$. By Remark \ref{rem:step3} and Remark \ref{path}, there is no directed path from $a_{23'}$ to $c_{1(2n+1)}, a_{12'}, a_{1(2n+1)'}$ and $a_{21'}$. This implies that $b_k$ must be $c_{23}$. Clearly, there is no directed path of length at least two from $a_{23'}$ to $c_{23}$. The shortcut $X$ must be a clique of size  three, which is a contradiction to the length of a shortcut. Thus, $b_1$ cannot be $a_{23'}$.			
		\end{itemize}
			Hence, $c_{12}$ cannot be $b_k$. 

		\item[-] Case 2: For $2\le i\le 2n-1$, $b_k = c_{i(i+1)}$. The vertex $b_1$ must be among the vertices $a_{i(i-1)'}$, $a_{i(i+1)'}$, $a_{(i+1)i'}$ and $a_{(i+1)(i+2)'}$, and the vertex $b_{k-1}$ must be among the vertices $c_{(i+1)(i+2)}$, $a_{i(i-1)'}$, $a_{i(i+1)'}$, $a_{(i+1)i'}$ and $a_{(i+1)(i+2)'}$. We handle this case in the following four subcases:
		\begin{itemize}
			\item Subcase 2.1: $b_1 = a_{i(i-1)'}$. By Remark \ref{rem:step3} and Remark \ref{path}, there is no directed path from $a_{i(i-1)'}$ to $a_{(i+1)i'}$, $a_{(i+1)(i+2)'}$ and $c_{(i+1)(i+2)}$. Thus, $b_{k-1}$ must be $a_{i(i+1)'}$. Clearly, there is no directed path of length at least two from $a_{i(i-1)'}$ to $a_{i(i+1)'}$. The shortcut $X$ must be of length two, which is a contradiction. Hence, $b_1$ cannot be $a_{i(i-1)'}$.
			\item Subcase 2.2: $b_1 = a_{i(i+1)'}$. By Remark \ref{rem:step3} and Remark \ref{path}, there is no directed path from $a_{i(i+1)'}$ to $a_{i(i-1)'}$, $a_{(i+1)i'}$, $a_{(i+1)(i+2)'}$ and $c_{(i+1)(i+2)}$. This implies that there is no directed path from $a_{1(2n+1)'}$ to any possible vertices of $b_{k-1}$. Thus, $b_1$ cannot be $a_{i(i+1)'}$.
			\item Subcase 2.3: $b_1 = a_{(i+1)i'}$. By Remark \ref{rem:step3} and Remark \ref{path}, there is no directed path from $a_{(i+1)i'}$ to the vertices $a_{i(i-1)'}$, $a_{i(i+1)'}$. The only possible shortcut is $\langle a_{(i+1)i'}, a_{(i+1)(i+2)'}, c_{(i+1)(i+2)}, c_{i(i+1)} \rangle$. However, these vertices induces a clique in $L(\mu(C_{2n+1}))$ (see Fig. \ref{ind_sub_22}). Thus, $b_1$ cannot be $a_{(i+1)i'}$.
			
				\begin{figure}[h]
				\centering
				
				\begin{tikzpicture}[scale=0.8]
					\begin{scope}[decoration={markings,mark=at position 0.5 with {\arrow[thick]{>}}}]
						\vertex (1) at (0,0) [label=left:$\scriptstyle{a_{(i+1)i'}}$] {};
						\vertex (2) at (2,0) [label=below:$\scriptstyle{a_{(i+1)(i+2)'}}$] {}; 
						\vertex (3) at (4,0) [label=below:$\scriptstyle{c_{(i+1)(i+2)}}$] {};
						\vertex (4) at (6,0) [label=right:$\scriptstyle{c_{i(i+1)}}$] {};

						\path 
						(1) edge[postaction={decorate}] (2)
						(2) edge[postaction={decorate}] (3)
						(3) edge[postaction={decorate}] (4)
						(1) edge[bend right = 30,postaction={decorate}] (3)
						(1) edge[bend left = 20,postaction={decorate}] (4)
						(2) edge[bend right = 30,postaction={decorate}] (4);
					\end{scope}
				\end{tikzpicture}
				\caption{The subgraph of $L(\mu(C_{2n+1}))$ induced by $a_{(i+1)i'}, a_{(i+1)(i+2)'}, c_{(i+1)(i+2)}$ and $c_{i(i+1)}$}
				\label{ind_sub_22}			
			\end{figure}
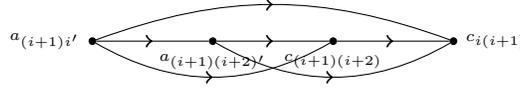
			\item Subcase 2.4: $b_1 = a_{(i+1)(i+2)'}$. By Remark \ref{rem:step3} and Remark \ref{path}, there is no directed path from $ a_{(i+1)(i+2)'}$ to the vertices $a_{i(i-1)'}$, $a_{i(i+1)'}$ and $a_{(i+1)i'}$. Thus, $b_{k-1}$ must be $c_{(i+1)(i+2)}$. Clearly, there is no directed path of length at least two from $a_{(i+1)(i+2)'}$ to $c_{(i+1)(i+2)}$. The shortcut $X$ must be a clique of size three, which is a contradiction. Thus, $a_{(i+1)(i+2)'}$ cannot be $b_1$.
		\end{itemize}
		Hence, $c_{i(i+1)}$ cannot be $b_k$, for $1\le i \le 2n-1$.
		
		\item [-] Case 3: $b_k = c_{1(2n+1)}$.  The vertex $b_1$ must be one among the vertices $a_{12'}$, $a_{1(2n+1)'}$, $a_{(2n+1)1'}$ and $a_{(2n+1)(2n)'}$ and the vertex $b_{k-1}$ must be among the vertices $c_{(2n)(2n+1)}$, $a_{12'}$, $a_{1(2n+1)'}$, $a_{(2n+1)1'}$ and $a_{(2n+1)(2n)'}$. If $b_1 \in \{a_{12'}, a_{1(2n+1)'}\}$ then, by Remark \ref{rem:step3} and Remark \ref{path}, there is no directed path from the vertex $b_1$ to any possible vertices of $b_{k-1}$. Thus, $b_k$ can neither be $a_{12'}$ nor $a_{1(2n+1)'}$. Further, if $b_1 = a_{(2n+1)1'}$ then the only possible shortcut is $\langle a_{(2n+1)1'}, a_{(2n+1)(2n)'}, c_{(2n)(2n+1)}, c_{1(2n+1)}\rangle$. However, these vertices induces a clique in $L(\mu(C_{2n+1}))$ (see. Fig. \ref{ind_sub_21}). Thus, $b_1$ cannot be $a_{(2n+1)1'}$.
		
		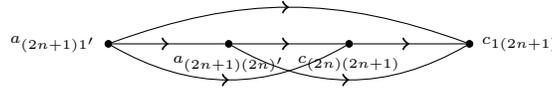
\begin{figure}[h]
			\centering				
			\begin{tikzpicture}[scale=0.8]
				\begin{scope}[decoration={markings,mark=at position 0.5 with {\arrow[thick]{>}}}]
					\vertex (1) at (0,0) [label=left:$\scriptstyle{a_{(2n+1)1'}}$] {};
					\vertex (2) at (2,0) [label=below:$\scriptstyle{a_{(2n+1)(2n)'}}$] {}; 
					\vertex (3) at (4,0) [label=below:$\scriptstyle{c_{(2n)(2n+1)}}$] {};
					\vertex (4) at (6,0) [label=right:$\scriptstyle{c_{1(2n+1)}}$] {};

					\path 
					(1) edge[postaction={decorate}] (2)
					(2) edge[postaction={decorate}] (3)
					(3) edge[postaction={decorate}] (4)
					(1) edge[bend right = 30,postaction={decorate}] (3)
					(1) edge[bend left = 20,postaction={decorate}] (4)
					(2) edge[bend right = 30,postaction={decorate}] (4);
				\end{scope}
			\end{tikzpicture}
			\caption{The subgraph of $L(\mu(C_{2n+1}))$ induced by $a_{(2n+1)1'}, a_{(2n+1)(2n)'}, c_{(2n)(2n+1)}$ and $c_{1(2n+1)}$}
			\label{ind_sub_21}			
		\end{figure} 
		If $b_1 = a_{(2n+1)(2n)'}$ then there is no directed path from $b_1$ to any of the vertices $a_{(2n+1)1'}, a_{12'}$ and $a_{1(2n+1)'}$. Thus, $b_{k-1}$ must be $c_{(2n)(2n+1)}$. Clearly, there is no directed path of length at least two from $b_1$ to $c_{(2n)(2n+1)}$. Thus, the shortcut $X$ forms a clique of size three, which is a contradiction. Hence, $b_1$ cannot be $a_{(2n+1)(2n)'}$.		
		
		\item[-] Case 4: $b_k = c_{(2n)(2n+1)}$. Since $c_{(2n)(2n+1)}$ is a source in the orientation $D$ when restricted to the induced subgraph $L(C_{2n+1})$. The vertex $b_{k-1}$ must be a vertex from $L(B_{2n+1})$. Thus, the vertices of $b_1$ and $b_{k-1}$ must be among the vertices $a_{(2n)(2n-1)'}$, $a_{(2n)(2n+1)'}$,  $a_{(2n+1)1'}$ and $a_{(2n+1)(2n)'}$. By Remark \ref{rem:step3} and Remark \ref{path}, if $b_1 \in \{a_{(2n)(2n-1)'}, a_{(2n+1)1'}\}$ then the shortcut $X$ must be a clique size three, which is a contradiction. If $b_1 \in  \{a_{(2n)(2n+1)'}, a_{(2n+1)(2n)'}\} $ then there is no directed path of length at least two from $b_1$ to $c_{(2n)(2n+1)}$. Hence, $b_k$ cannot be $c_{(2n)(2n+1)}$.
	\end{itemize} 
		
	Since none of the vertices of $L(C_{2n+1})$ is $b_k$, from the statement proved under `claim', we get that the shortcut $X$ contains only vertices of $L(B_{2n+1})$. Clearly, the orientation $D$ restricted to $L(B_{2n+1})$ does not contain any shortcut. Hence, the orientation $D$ contains no shortcuts.	 \qed
\end{proof}

Hence, from Lemma \ref{acyclic} and Lemma \ref{semi}, we have the following theorem.

\begin{theorem}\label{mu_cn}
	For $n \ge 2$, the line graph $L(\mu(C_{2n+1}))$ is word-representable.
\end{theorem}
               
\section{Concluding Remarks}       

In this paper, we showed that if a graph on at least $5$ vertices contains $K_4$ as a subgraph then its line graph is non-word-representable. Consequently, this led us to focus on non-word-representable graphs that are $K_4$-free, as we aim to address the open problem of whether the line graph of a non-word-representable graph is always non-word-representable. 

The class of  graphs that we have considered in this work  is  the Mycielski graphs of odd cycles, which are known to be triangle-free and non-word-representable. We  showed that their corresponding line graphs are, in fact, word-representable. Further, in view of   remarks \ref{con_1} and \ref{con_2}, we raise the following questions: 
\begin{enumerate}
	\item Is the line graph of a triangle free graphs always word-representable?
	\item Does there exist  a non-word-representable graph with clique number three whose line graph is word-representable?
\end{enumerate}

The answers to the above-mentioned two questions could potentially lead to a complete characterization of word-representable line graphs.

\end{document}